\documentclass[
 amsmath,amssymb,
  onecolumn 
]{revtex4-2}

\usepackage[a4paper,margin=1.2in]{geometry}

\usepackage{graphicx}%
\usepackage{amsmath,amssymb,amsfonts}%
\usepackage{amsthm}%
\usepackage{mathrsfs}%
\usepackage{xcolor}%
\usepackage{textcomp}%

\usepackage{tikz}
\usetikzlibrary{calc}
\usepackage{enumitem}

\usepackage{amscd,amsthm,amsfonts,latexsym,amssymb}

\usetikzlibrary{patterns}
\usepackage{braket}

\newtheorem{theorem}{Theorem} [section]
\newtheorem{lemma}[theorem]{Lemma}
\newtheorem{corollary}[theorem]{Corollary}
\newtheorem{proposition}[theorem]{Proposition}
\newtheorem{example}[theorem]{Example}

\newtheorem{definition}[theorem]{Definition}
\newtheorem{remark}[theorem]{Remark}
\newtheorem{axiom}{Axiom}

\newcommand{\norm}[1]{||} 

\newcommand{\Ff}{{\mathcal F}}  

\newcommand{\Aa}{{\mathcal A}}  
\newcommand{\Bb}{{\mathcal B}}

\renewcommand{\phi}{\varphi}

\newcommand{\si}{\sigma}

\newcommand{\Om}{\Omega}

\begin{document}

\title{Endogenous Measures and Refinement Dynamics on Finite $\sigma$-Algebra Systems}

\author{Paul Baird}
\affiliation{Laboratoire de Math\'ematiques de Bretagne Atlantique, Universit\'e de Bretagne Occidentale, Brest, France}

\begin{abstract}
We study systems of $\sigma$-algebras ordered by refinement and introduce the
notion of an endogenous probability measure, invariant under admissible
refinement transformations. We prove existence and structural properties
of such measures on finite systems and show how refinement
operators induce a natural dynamical structure on the lattice of
$\sigma$-algebras.
\end{abstract}

\maketitle

\section{Introduction}\label{sec1}

Systems of $\sigma$-algebras on a finite set, ordered by refinement,
form a natural partially ordered structure closely related to the
lattice of partitions of the set. Such systems arise whenever one
considers progressively finer resolutions of a state space.
We study probability measures that remain invariant under admissible
refinement transformations. When this invariance condition is imposed,
the refinement structure and the measure become mutually constrained.
This leads to the notion of an \emph{endogenous measure}: a probability
measure that is stable under the refinement system that it itself
supports.

The main result establishes the existence and structural properties
of endogenous measures on finite refinement systems. In particular,
compatible families of refinements admit self-consistent invariant
measures that are minimal with respect to the refinement order on
$\sigma$-algebras (Theorem~\ref{th:endogenous}).

Logical and combinatorial approaches to physical structure have a long
history. The lattice-theoretic formulation of quantum propositions was
introduced by Birkhoff and von Neumann \cite{BvN}. Discrete causal
structures have been investigated in causal-set models of spacetime
\cite{BLS}, while the consistent-histories approach studies compatible
families of quantum events \cite{Gr}. The framework developed here is
different in spirit: it is based on directed systems of refining
$\sigma$-algebras together with the existence of self-consistent
invariant measures on such systems. Background on invariant measures
and ergodic theory can be found in \cite{Pe}.

This perspective brings together ideas that normally appear in
different contexts. The refinement order on $\sigma$-algebras provides
a natural notion of progressively finer distinctions, while
compatibility of refinements gives rise to a partial ordering
reminiscent of causal structures. At the same time, families of
mutually compatible refinements play a role analogous to consistent
histories in quantum theory \cite{Gr}. Within this framework
probability measures are not specified externally but arise as
invariant structures determined by the refinement dynamics itself.

The paper is organised as follows.  Section~\ref{sec2} introduces commuting
refinements and the associated notion of compatibility.  Section~\ref{sec3}
develops the concept of endogenous measures and proves the existence
and structural properties of self-consistent invariant measures on
finite refinement systems.  Section~\ref{sec4} studies refinement dynamics and
shows how partial ordering and causal interpretation emerge from the
structure of refinements.  Section~\ref{sec6} presents a minimal toy model
illustrating the basic mechanisms.  Section~\ref{sec7} outlines a general
theory of refinement dynamics and discusses how sequences of
refinements generate larger-scale structures.

\section{Foundations: Information, Distinction, and Refinement}\label{sec2}

We consider a finite set $\Omega$ together with the trivial $\sigma$-algebra
$$
\mathcal F_0 = \{\emptyset,\Omega\}.
$$
This represents a situation in which no distinctions are available: all elements of $\Omega$ are informationally indistinguishable.

More generally, informational structure will be represented by $\sigma$-algebras on $\Omega$. Passing from one $\sigma$-algebra to a finer one corresponds to the introduction of new distinctions.

\subsection{Distinctions as $\sigma$-algebras}

A $\sigma$-algebra $\Ff$ on $\Om$ is a family of subsets closed under complements and countable unions. An \emph{atom} of $\Ff$ is a nonempty set $A \in \Ff$ such that no proper nonempty subset of $A$ lies in $\Ff$. 
The atoms, when they exist, form a partition of the atomic part of $\Om$, i.e. of 
$$
\Om_{\rm at}:= \bigcup\{ A\subseteq\Om : A \text{ is an atom of } \Ff\}. 
$$

In general $\sigma$-algebras (for example the Borel $\sigma$-algebra on $\mathbb{R}$) may have no nonempty atoms; the $\sigma$-algebra can be atomless.
Because $\Omega$ is finite, however, every $\sigma$-algebra $\mathcal F$
is uniquely determined by its partition into atoms.
Thus the refinement order on $\sigma$-algebras corresponds exactly to the
refinement order on partitions of $\Omega$.
Throughout the paper we freely move between these equivalent viewpoints.
\subsection{Events as refinements} 
An \emph{event} is modelled as a refinement
$$
\Ff \longrightarrow \Ff'
$$
where $\Ff'$ has strictly finer atoms than $\Ff$. Intuitively

\begin{itemize}[leftmargin=2em]
\item An event ``carves" one or more of the existing atoms into smaller pieces.
\item A chain of refinements
$$
\Ff_0 \subset \Ff_1 \subset \cdots \subset \Ff_n
$$
records the \emph{growth of distinguishability} in the universe.
\item Such chains will later be interpreted as \emph{temporal sequences}: time is nothing but the accumulation of distinctions.
\end{itemize}

This notion of an event contains the seed of both classical measurements and quantum measurements, but without invoking Hilbert spaces or operators.

\subsection{Choice and measure extension} 
We attach to each $\sigma$-algebra $\Ff$ a probability measure $\mu$ defined on $\Ff$. At this stage the measure is introduced formally, as if externally specified, in order to define consistent and inconsistent refinements. In Section 3, this assumption will be relaxed: the measure itself will be required to arise endogenously from the refinement structure, becoming the fixed point of a self-consistency condition between events and probabilities.

A refinement $\Ff \longrightarrow \Ff'$ requires extending $\mu$ to a new measure $\mu'$ on $\Ff'$. This extension may be:
\begin{itemize}[leftmargin=2em]
\item \emph{unique}--classically deterministic;
\item \emph{non-unique}--genuinely indeterministic, corresponding to intrinsic indeterminancy;
\item \emph{impossible}--indicating that $\Ff'$ is incompatible with $\Ff$ (a key precursor to contextuality and non-commutativity). 
\end{itemize}
Thus:

\begin{definition} (Consistent Measure Extension). 
A refinement $\Ff \longrightarrow \Ff'$ is \emph{consistent} if there exists at least one probability measure $\mu'$ on $\Ff'$ whose restriction to $\Ff$ equals the prior measure $\mu$. 
\end{definition}  

The number of allowed extensions encodes the degree of determinacy. A refinement for which many extensions are possible represents a genuine choice -- a point where the universe’s informational unfolding is not predetermined.

In Section 3 we will show that chains of consistent measure extensions naturally form causal (time-like) orderings, whereas incompatible $\sigma$-algebras give rise to space-like separability or its failure.

\subsection{Commutativity and the emergence of independence} 
The $\sigma$-algebras on a finite set $\Omega$, ordered by inclusion
form a lattice in which joins correspond to joint refinements and
meets to common coarsenings (i.e.\ $\sigma$-algebras obtained by
merging atoms of a finer partition). Given two relations $R$ and $S$ on a set $\Om$, we define their composition $R\circ S$ by $(x,z) \in R \circ S$ if and only if there exists $y \in \Om$ such that $(x,y)\in S$ and $(y,z) \in R$.  Every finite $\sigma$-algebra determines an equivalence relation on $\Omega$ by declaring two elements equivalent when they lie in the same atom.

\begin{definition}(Commuting Refinements) \label{CR} 
Let $\Aa$ and $\Bb$ be $\sigma$-algebras on the same finite $\Om$, and let $x \sim_{\Aa} y$ (resp. $x\sim_{\Bb} y)$ denote the equivalence relation defined by membership in the same atom of $\Aa$ (resp. $\Bb$).  We say that $\Aa$ and $\Bb$ \emph{commute}, or are \emph{compatible}, if 
$$
\sim_{\Aa} \circ \sim_{\Bb} = \sim_{\Bb} \circ \sim_{\Aa}\,.
$$
In this case the corresponding refinements are called \emph{commuting}. 
\end{definition}

\noindent Note that the composition of two equivalence relations need not be an equivalence relation, but if they commute then it is. 

In what follows, we will use the term \emph{commuting refinements} for pairs of $\sigma$-algebras whose associated equivalence relations commute. In physical language we may also call such refinements \emph{compatible}, as they represent distinctions that can coexist within a single informational state.

\medskip

\begin{example}\label{ex:NC} Consider the three $\sigma$-algebras $\Aa, \Bb, \Bb'$ defined on the set $\Om = \{1,2,3,4\}$ by $\Aa = \{ \{ 1,2\} , \{ 3,4\} \}$, $\Bb = \{ \{ 1,3\} , \{ 2,4\} \}$ and $\Bb' = \{ \{ 1,3,4\} , \{ 2\}\}$.  Then one can easily check that $\Aa$ and $\Bb$ commute.  On the other hand $2\sim_{\Aa} 1$ and $1 \sim_{\Bb'} 3$, hence $2$ is related to $3$ in $\sim_{\Bb'} \circ \sim_{\Aa}$. For $\sim_{\Aa} \circ \sim_{\Bb'}$ we need $y$ with $2\sim_{\Bb'} y$ and $y\sim_{\Aa} 3$. But $2\sim_{\Bb'} y$ forces $y =2$ but $2\sim_{\Aa} 3$ is false. So no such $y$ exists and $(2,3)$ is not in $\sim_{\Aa} \circ \sim_{\Bb'}$. So $\sim_{\Aa} \circ \sim_{\Bb'} \neq \sim_{\Bb'} \circ \sim_{\Aa}$, and $\Aa$ and $\Bb'$ are non-commuting refinements.
\end{example}

\noindent Commutativity represents the possibility of making both sets of distinctions simultaneously. This is the purely measure-theoretic precursor to:
\begin{itemize}[leftmargin=2em]
\item commuting observables in quantum mechanics,
\item space-like separated events in relativity,
\item local independence in statistical physics.
\end{itemize}
 Non-commutativity implies contextuality: distinctions acquire meaning only relative to the $\si$-algebra in which they are realised. Thus, non-commutativity signals \emph{contextuality}, \emph{order-dependence}, or \emph{causal entanglement}.

In the next section we study refinement sequences and their associated measure structures. These will lead naturally to a partial ordering of refinements and to the notion of endogenous probability measures.

\section{Endogenous Measures and Dynamical Self-Consistency}\label{sec3} 
In classical probability theory the measure is simply given -- either as an objective frequency assignment or as an agent’s prior. If the $\sigma$-algebra encodes distinctions that ``exist" only insofar as they can be stably made, then a measure should not be an externally imposed weighting of possibilities, but rather the fixed point of a self-consistency condition that couples:

\begin{enumerate}[leftmargin=2em]
\item the $\sigma$-algebra of meaningful distinctions, and
\item	the measure that makes those distinctions dynamically stable.
\end{enumerate}

Our goal in this section is to formalise this idea.

\subsection{The basic principle: refinements stabilise probabilities,
and probabilities stabilise refinements} 
Let $\mathfrak{F}$ denote the class of $\sigma$-algebras over a base set $\Om$ 
ordered by refinement:
$$
\Aa \preceq \Bb \quad \text{iff} \quad \Bb \text{ is a refinement of } \Aa
$$

Given a $\sigma$-algebra $\mathcal F$, a probability measure $\mu$
assigns probabilities to the measurable sets $A\in\mathcal F$,
and in the finite case is determined by its values on the atoms of
$\mathcal F$. However, in the physical picture we are developing, not every refinement $\Ff' \succeq \Ff$ should be meaningful. A refinement is meaningful only if:
\begin{itemize}[leftmargin=2em]
\item it corresponds to distinctions that can persist,
\item it does not generate non-commutating refinements, and
\item it is recognised by a measure that renders those distinctions dynamically coherent.
\end{itemize}
Thus the measure and the $\sigma$-algebra restrict one another.

\begin{definition}\label{def:invariant}
(Self-consistency of a measure and $\sigma$-algebra).
A pair $(\mathcal F,\mu)$ is \emph{self-consistent} if:

\begin{enumerate}[leftmargin=2em]
\item (Compatibility)
The refinements generating $\mathcal F$ satisfy the commutativity
condition of Definition~\ref{CR}.

\item (Invariance under $\mathcal F$-preserving transformations)
The measure $\mu$ is invariant under transformations that preserve the
$\sigma$-algebra $\mathcal F$.  Formally, if
$T:\Omega\to\Omega$ satisfies
\[
T^{-1}A \in \mathcal F \qquad\text{for all } A\in\mathcal F ,
\]
then
\[
\mu(T^{-1}A)=\mu(A)
\qquad\text{for all }A\in\mathcal F .
\]

\item (Refinement consistency)
If $\mathcal F'$ strictly refines $\mathcal F$, then $\mu$ extends to
$\mathcal F'$ only if the refinement respects the commutativity
conditions of Definition~\ref{CR}.

\end{enumerate}
\end{definition}

The last condition is the heart of the endogenous viewpoint: refinements exist only when the measure can meaningfully distinguish the refinements. If it cannot, they are not physically realised distinctions.

\subsection{The refinement–measure map}
For each $\sigma$-algebra $\mathcal F$ define
\[
\mathcal M(\mathcal F)
= \{\,\mu \mid (\mathcal F,\mu)\text{ is self-consistent}\,\},
\]
the set of probability measures that are self-consistent with
$\mathcal F$.

Conversely, for each probability measure $\mu$ define
\[
\mathfrak F(\mu)
= \{\,\mathcal F\in\mathfrak F \mid (\mathcal F,\mu)\text{ is
self-consistent}\,\},
\]
the class of $\sigma$-algebras that are compatible with $\mu$. Thus we obtain a pair of correspondences
\[
\mathcal F \longmapsto \mathcal M(\mathcal F),
\qquad
\mu \longmapsto \mathfrak F(\mu),
\]
relating $\sigma$-algebras and probability measures.
\begin{definition}\label{def:endogenous}
(Endogenous measure).
A probability measure $\mu$ is called \emph{endogenous} if there exists
a $\sigma$-algebra $\mathcal F$ such that
\[
\mu\in\mathcal M(\mathcal F)
\quad\text{and}\quad
\mathcal F\in\mathfrak F(\mu).
\]
Equivalently, the pair $(\mathcal F,\mu)$ is a fixed point of the
refinement--measure correspondence.
\end{definition}
This identifies a distinguished $\sigma$-algebra:
$$
\mathcal{F}^{\ast} := \bigcap_{\mu^{\ast}} \mathfrak{F}(\mu^{\ast}),
$$
where the intersection is taken over all endogenous measures $ \mu^{\ast} $.
The construction mirrors how in quantum mechanics the Born rule and the lattice of projections mutually constrain one another—but here we remain entirely within classical $\sigma$-algebraic probability.

\subsection{Stability criteria for endogenous measures}
To make this concept operational, we need a way to test whether a candidate measure $\mu $ is endogenous.
We collect a minimal set of axioms from which the framework can be built:

\medskip

\begin{axiom}[Commutativity domain]\label{ax:A}
There exists a nonempty set $\mathfrak C \subset \mathfrak F$ of
$\sigma$-algebras such that any two elements of $\mathfrak C$ commute
in the sense of Definition~\ref{CR}. Moreover, $\mathfrak C$ is closed
under coarsening: if $\mathcal F\in\mathfrak C$ and
$\mathcal G\prec\mathcal F$, then $\mathcal G\in\mathfrak C$.
\end{axiom}

\begin{axiom}[Invariant symmetries]\label{ax:B}
For each $\mathcal F\in\mathfrak C$ there exists a nontrivial
measure-preserving transformation
$T:\Omega\to\Omega$ satisfying
\[
T^{-1}A\in\mathcal F
\quad\text{and}\quad
\mu(T^{-1}A)=\mu(A)
\qquad\text{for all }A\in\mathcal F .
\]
Equivalently, the group of $\mathcal F$-automorphisms preserving $\mu$
is nontrivial.
\end{axiom}

\begin{axiom}[Refinement stability]\label{ax:C}
Let $\mathcal F\in\mathfrak C$ and let $\mu$ be a measure on $\mathcal F$.
If $\mathcal F'\succeq\mathcal F$ is a refinement such that
$\mathcal F'\notin\mathfrak C$, then $\mu$ admits no extension to a
measure on $\mathcal F'$ that preserves the invariance conditions of
Axiom~\ref{ax:B}.
\end{axiom}

This latter axiom will later translate into spatial compatibility
and temporal ordering: refinements that preserve commutativity
correspond to spacelike separation, while refinements that violate
it generate timelike or causal relations.

\begin{definition}\label{def:degenerate}
A probability measure $\mu$ on a finite $\sigma$-algebra $\mathcal F$
is called \emph{degenerate} if it is supported on a single atom of $\mathcal F$,
that is, if there exists an atom $A\in\mathcal F$ such that $\mu(A)=1$.
We call $\mu$ \emph{nondegenerate} if $\mu(A)\in(0,1)$ for every atom $A$.
\end{definition}

\subsection{Existence and uniqueness of endogenous measures}
The refinement operators introduced above generate a semigroup
$\mathcal R$ acting on the set $\mathfrak C$ of admissible
$\sigma$-algebras.
An endogenous measure may therefore be interpreted as a probability
measure that is invariant under this refinement action.

\begin{theorem}[Existence of endogenous invariant measures]\label{th:endogenous}
Let $\Omega$ be a finite set and let $\mathfrak F$ denote the set of
$\sigma$-algebras on $\Omega$ ordered by refinement.
Let $\mathfrak C\subseteq\mathfrak F$ be a nonempty compatibility domain
closed under coarsening, and let I(F) denote the (nonempty) set of invariant probability measures
Then the refinement–measure correspondence admits a fixed point:
there exists a pair $(\mathcal F^\ast,\mu^\ast)$ with
$\mathcal F^\ast\in\mathfrak C$ and $\mu^\ast\in I(\mathcal F^\ast)$ such that

\begin{enumerate}[label=(\roman*),leftmargin=2.5em]

\item $(\mathcal F^\ast,\mu^\ast)$ is maximal among admissible
self-consistent pairs under refinement;

\item $\mathcal F^\ast$ is minimal among endogenous $\sigma$-algebras
supporting $\mu^\ast$;

\item (Uniqueness up to symmetry.)
Let $G\leq\mathrm{Sym}(\Omega)$ be a group acting on $\Omega$.
The action induces a natural action on $\sigma$-algebras by
\[
g\cdot A = \{\, g(x) : x\in A \,\}, \qquad A\subseteq\Omega .
\]
Assume that $G$ preserves the commutativity condition of
Definition~\ref{CR} and acts transitively on the atoms of any
minimal endogenous $\sigma$-algebra.
Then the endogenous pair $(\mathcal F^\ast,\mu^\ast)$ is unique
up to the action of $G$.
\end{enumerate}
\end{theorem}
The action of $G$ on $\Omega$ induces actions on both
$\sigma$-algebras and measures via pushforward.
In the finite setting considered here, invariant measures always exist:
given any probability measure $\mu$ on $\mathcal F$, averaging over the
finite automorphism group $\mathrm{Aut}(\mathcal F)$ produces an
$\mathrm{Aut}(\mathcal F)$-invariant measure.
Thus the assumption $I(\mathcal F)\neq\varnothing$ is automatically satisfied.

\begin{remark}
The theorem may be viewed as an invariant-measure result for the
semigroup of refinement operators acting on the lattice of
$\sigma$-algebras.
\end{remark}

The proof proceeds in three stages: existence of invariant measures,
minimality of the supporting $\sigma$-algebra, and uniqueness up to symmetry. We first establish an elementary lemma.

\begin{lemma}\label{lem:maxsigma}
Let $\Omega$ be finite and let $\mathfrak C$ be a nonempty family of
$\sigma$-algebras on $\Omega$ that is closed under coarsening.
Then $\mathfrak C$ contains at least one $\preceq$-maximal element,
i.e.\ there exists $\mathcal F^{\max}\in\mathfrak C$ such that there is no
$\mathcal G\in\mathfrak C$ with $\mathcal F^{\max}\subsetneq \mathcal G$.
\end{lemma}

\begin{proof}
Since $\Omega$ is finite, the set of $\sigma$-algebras on $\Omega$ is finite.
Hence any strictly increasing chain in $\mathfrak C$ must terminate, so
$\mathfrak C$ contains a maximal element under inclusion.
\end{proof}

\medskip
\noindent\textbf{(i) Existence.}
By Lemma~\ref{lem:maxsigma}, choose a $\preceq$-maximal
$\mathcal F^{\max}\in\mathfrak C$.
By hypothesis $I(\mathcal F^{\max})\neq\varnothing$, so pick
$\mu^{\max}\in I(\mathcal F^{\max})$.

We claim that $(\mathcal F^{\max},\mu^{\max})$ is endogenous.
Indeed, if there existed a strict refinement
$\mathcal G\in\mathfrak C$ with $\mathcal F^{\max}\subsetneq\mathcal G$
such that $\mu^{\max}$ extended to some $\nu\in I(\mathcal G)$,
then $\mathcal G$ would contradict the maximality of $\mathcal F^{\max}$.
Hence no admissible refinement extends $(\mathcal F^{\max},\mu^{\max})$,
and the pair cannot be extended to any strictly finer admissible
$\sigma$-algebra.

\begin{corollary}
In the finite setting, endogenous pairs exist whenever each
$\mathcal F\in\mathfrak C$ admits at least one invariant measure.
\end{corollary}

\begin{remark}
In the finite setting the existence of endogenous pairs reduces to a maximality
argument on the (finite) refinement poset. In infinite settings one expects to
replace this step by compactness or Zorn-type arguments in appropriate weak
topologies.
\end{remark}

Prior to establishing minimality, we require two lemmas. 

\begin{lemma}\label{lem:min-endog}
Assume $\Omega$ is finite and that at least one endogenous pair exists.
Let
\[
\mathcal E := \{\mathcal F\in \mathfrak C : \exists\,\mu\in I(\mathcal F)
\text{ such that } (\mathcal F,\mu)\text{ is endogenous}\}.
\]
Then $\mathcal E$ contains a $\preceq$-minimal element; i.e.\ there exists
$\mathcal F^\ast\in\mathcal E$ such that there is no
$\mathcal G\in\mathcal E$ with $\mathcal G\prec \mathcal F^\ast$.
\end{lemma}

\begin{proof}
Because $\Omega$ is finite, the set of $\sigma$-algebras on $\Omega$ is finite,
hence $\mathcal E$ is finite. Any nonempty finite partially ordered set contains a minimal element.
\end{proof}

\begin{definition}\label{def:admissible}
Let $\mathcal F\in\mathfrak C$. A refinement $\mathcal F\preceq\mathcal H$
is called \emph{admissible} if $\mathcal H\in\mathfrak C$.
\end{definition}

\begin{lemma}[Restriction preserves admissible self-consistency]
\label{lem:restrict}
Let $(\mathcal F,\mu)$ be self-consistent with $\mathcal F\in\mathfrak C$,
and let $\mathcal G\prec\mathcal F$ with $\mathcal G\in\mathfrak C$.
Then $(\mathcal G,\mu|_{\mathcal G})$ is self-consistent with respect to
admissible refinements (Definition~\ref{def:admissible}).
In particular, $\mu|_{\mathcal G}\in I(\mathcal G)$.
\end{lemma}

\begin{proof}
\emph{Refinement commutativity}: Since $\mathcal G\in\mathfrak C$, commutativity holds
by definition of $\mathfrak C$.

\noindent \emph{Dynamical stability}: If $T$ preserves the $\mathcal G$-atoms, then it also
preserves the $\mathcal F$-atoms up to unions of $\mathcal F$-atoms, hence
$\mu(T^{-1}A)=\mu(A)$ for all $A\in\mathcal G$ follows from invariance of
$\mu$ on $\mathcal F$.

\noindent \emph{Admissible refinement stability}: Let $\mathcal H\succeq\mathcal G$ be an
admissible refinement, so $\mathcal H\in\mathfrak C$.
If $\mu|_{\mathcal G}$ extends to $\mathcal H$, then by Axiom~\ref{ax:C}
this is possible only when commutativity is preserved.
Conversely, if commutativity is preserved, admissibility means $\mathcal H$
lies in the same domain in which consistent extensions are allowed.
\end{proof}

\noindent\textbf{(ii) Minimality.}

By Lemma~\ref{lem:min-endog}, choose an endogenous pair
$(\mathcal F^\ast,\mu^\ast)$ such that $\mathcal F^\ast$ is minimal among
endogenous $\sigma$-algebras in $\mathfrak C$.
Suppose by contradiction that there exists a strict coarsening
$\mathcal G\prec\mathcal F^{\ast}$ such that $\mu^{\ast}$ is nondegenerate
on $\mathcal G$ (i.e.\ assigns strictly positive probability to every atom
of $\mathcal G$).
Because $\mathfrak C$ is closed under coarsening,
$\mathcal G\in\mathfrak C$.

By Lemma~\ref{lem:restrict}, the restricted pair
$(\mathcal G,\mu^{\ast}|_{\mathcal G})$ is self-consistent and
$\mu^{\ast}|_{\mathcal G}\in I(\mathcal G)$.
Applying the existence construction within $\mathcal G$
there therefore exists an endogenous pair supported on $\mathcal G$.
This contradicts the minimality of $\mathcal F^{\ast}$.
Hence no such $\mathcal G$ exists; equivalently every strict
coarsening of $\mathcal F^{\ast}$ makes $\mu^{\ast}$ degenerate.

\begin{remark} Intuitively $\mathcal F^{\ast}$ is the finest partition actually “seen’’ by $\mu^{\ast}$; coarsening loses support mass and thus is not faithful.
\end{remark} 

\begin{lemma}\label{lem:ergodic}
Let $\Omega$ be finite and let $G\leq \mathrm{Sym}(\Omega)$ act transitively
on the atoms of a $\sigma$-algebra $\mathcal F$.
Then there exists a unique $G$-invariant probability measure on $\mathcal F$,
namely the uniform measure on the atoms of $\mathcal F$.
\end{lemma}

\begin{proof}
Let $A_1,\dots,A_k$ be the atoms of $\mathcal F$.
If $\mu$ is $G$-invariant, then for any $g\in G$
\[
\mu(gA_i)=\mu(A_i).
\]
Since the action is transitive on the atoms, every atom can be mapped to
every other atom by some element of $G$, so all atoms must have equal
probability.
Thus $\mu(A_i)=1/k$ for all $i$, which uniquely determines $\mu$.
\end{proof}

\medskip
\noindent\textbf{(iii) Uniqueness up to symmetry.}
Assume the symmetry hypothesis stated in the theorem.
\medskip

$\bullet$	There is a (finite) group $G$ of bijections of $\Omega$ preserving the compatibility domain $\mathfrak{C}$ (i.e.\ for each $g\in G$ and $\mathcal F\in\mathfrak{C}$,
the pushforward $g(\mathcal F)\in\mathfrak{C}$).

\medskip

$\bullet$	The action of $G$ on the set of minimal endogenous $\si$-algebras is transitive.

\medskip

Let $(\mathcal F_1,\mu_1)$ and $(\mathcal F_2,\mu_2)$ be two endogenous pairs produced by Theorem \ref{th:endogenous}. By minimality and the transitivity assumption there exists $g\in G$ such that $g(\mathcal F_1)=\mathcal F_2$. Consider the pushforward measure $g_\sharp \mu_1$ on $\mathcal F_2$ given by $g_\sharp \mu_1(A)=\mu_1(g^{-1}A)$ for $A\in\mathcal F_2$. Because $\mu_1$ satisfied the invariance properties relative to $\mathcal F_1$ and $g$ preserves compatibility domain and the invariance structure, $g_\sharp \mu_1$ lies in $I(\mathcal F_2)$. But $\mu_2$ is an element of $I(\mathcal F_2)$ as well.  By Lemma~\ref{lem:ergodic}, the $G$-invariant probability measure on
$\mathcal F_2$ is unique.
Since both $\mu_2$ and $g_\sharp\mu_1$ are invariant under the induced
symmetry action, they must coincide.
This completes the proof of Theorem \ref{th:endogenous}.

\begin{remark} The final identification step commonly requires an ergodicity/irreducibility hypothesis on the action of the automorphism group on atoms; without such an assumption one only has equality up to an element of the convex set of invariant measures. The statement above therefore includes the natural extra assumption under which uniqueness (up to symmetry) holds. This mirrors physical situations where symmetry together with ergodicity single out a unique invariant measure (e.g. microcanonical ensembles under full permutation symmetry).
\end{remark} 

\begin{remark}[generalization to infinite $\Omega$ and measure-theoretic issues]
The finite-case arguments given above avoid the technical
measure-theoretic difficulties that arise in infinite settings. In the countably infinite or continuous case, existence of invariant measures and extension problems call for compactness arguments in appropriate weak-* topologies, use of Prokhorov’s theorem, and careful handling of $\si$-algebra generation. Zorn’s lemma and Krein–Milman style convex analysis (extreme points of invariant measure sets) become necessary. These generalizations are standard in ergodic theory and probability on Polish spaces and are omitted here for brevity; the finite-case results capture the essential combinatorial and algebraic structure
of the endogenous fixed-point construction. In the infinite case, the existence of invariant measures under compactness conditions follows by standard results in ergodic theory and dynamical systems \cite{BK, P}.
\end{remark}

\section{Refinement Dynamics and Dependency Structure}\label{sec4}
We now study the combinatorial structure generated by successive
refinements of $\sigma$-algebras on the finite set $\Omega$.
A refinement sequence
\[
\mathcal F_0 \subseteq \mathcal F_1 \subseteq \mathcal F_2 \subseteq \cdots
\]
forms a directed system in the poset of $\sigma$-algebras ordered by
refinement.
The elementary steps of this system are atomic refinements,
which subdivide individual atoms of the current partition.
These elementary refinements generate a dependency structure
that will be represented by a directed graph.

\subsection{Refinements as elementary operations} 
A \emph{refinement step} is a map
$$
\mathcal{F} \longrightarrow \mathcal{G}
$$
such that $ \mathcal{F} \subsetneq \mathcal{G} $ and the atoms of $ \mathcal{G} $ form a strict refinement of the atoms of $ \mathcal{F} $.
\begin{definition} (Atomic refinement).
A refinement $ \mathcal{F} \to \mathcal{G} $ is \emph{atomic} if exactly one atom of $\mathcal{F}$ is subdivided, and all others remain unchanged.
\end{definition} 
Atomic refinements therefore generate the directed system of
$\sigma$-algebras under consideration. Any finite refinement can
be decomposed into a finite sequence of such elementary steps.
\begin{proposition} 
Every finite refinement can be expressed (non-uniquely) as a finite sequence of atomic refinements.
\end{proposition} 
\begin{proof}
Every refinement of finite partitions can be obtained by
successively subdividing atoms.
Equivalently, if $P(\mathcal F)$ and $P(\mathcal G)$ are the atom
partitions with $P(\mathcal G)$ refining $P(\mathcal F)$,
one may split the atoms of $P(\mathcal F)$ one at a time until
$P(\mathcal G)$ is obtained.
\end{proof}
Each atomic refinement corresponds to the creation of a new distinction. Time will emerge from the partial order generated by these refinement relations.

\subsection{Refinement graph and induced partial order} 
Let $\mathfrak{R}$ denote the class of all atomic refinements
between successive $\sigma$-algebras in a refinement chain,
$$
R : \mathcal{F}_i \longrightarrow \mathcal{F}_{i+1}.
$$

\begin{definition}[Event graph]
Let $\mathfrak R$ be the set of atomic refinements.
Define a relation $R_1\prec R_2$ if in every refinement chain
containing both steps, $R_1$ occurs before $R_2$.
The event graph $\mathcal E$ is the directed graph whose nodes are
elements of $\mathfrak R$ with edges given by this relation.
\end{definition}

The relation $\prec$ captures the dependency structure of
refinement operations: $R_1\prec R_2$ precisely when the
distinction introduced by $R_2$ can only be defined after
that introduced by $R_1$.

\begin{figure}[htbp]
\centering
\begin{tikzpicture}[line width=0.3mm,black,scale=0.5]
\draw (-2, 1) -- (-1,3);
\draw (2,1) -- (1,3);
\draw (-2,-1) -- (-1, -3);
\draw (2,-1) -- (1, -3); 
\node at (-2.5,0) {$R_2$}; 
\node at (2.5, 0) {$R_3$};  
\node at (0, 4) {$R_1$};
\node at (0, -4) {$R_4$}; 
\end{tikzpicture}
\caption{An abstract event graph illustrating spacelike and timelike relations between refinement events.}
\label{fig:eventgraph}
\end{figure}
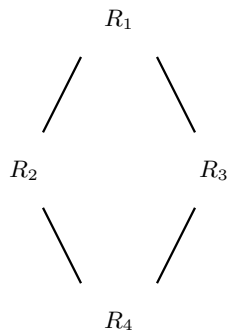

\begin{theorem} 
The relation $ \prec $ is a partial order. Thus $ \mathcal{E} $ is a directed acyclic graph (DAG).
\end{theorem} 
\begin{proof}
Reflexivity and transitivity follow directly from the definition
of refinement chains, while antisymmetry holds because two atomic
refinements cannot precede each other in all refinement chains.
\end{proof}

\noindent The event DAG is the emergent analogue of a causal set: a discrete
structure in which causal relations are represented by a partial
order on events \cite{BLS}. No background time is presupposed; temporal
order arises purely from the refinement structure: event $ R_1 $ precedes $ R_2 $ exactly when the distinctions introduced by $ R_2 $ are only definable after those introduced by $ R_1 $. In Figure \ref{fig:eventgraph} the vertical placement is schematic only. The diagram illustrates a partial order on refinement events: $R_1$ precedes both $R_2$ and $R_3$; $R_2$ and $R_3$ are spacelike-separated (incompatible); and $R_4$ succeeds both. No metric or geometric structure is implied. 

\subsection{Commutativity and independence}
Compatibility of refinements is determined by commutativity
of the associated partitions.
Compatible refinements can be performed in either order and
lead to the same joint refinement.

We now formalise the central idea: spacelike relations arise from commutative refinements, and timelike relations from non-commutative refinements.

Let $R_A : \mathcal F \to \mathcal F_A$ and $R_B : \mathcal F \to \mathcal F_B$ be two atomic refinements of the same $\sigma$-algebra.

\begin{definition}[Commutativity of refinements]
Refinements $R_A$ and $R_B$ are \emph{commutative} (or \emph{compatible}) if their associated equivalence relations $\!\sim_A$ and $\!\sim_B$ on $\Omega$ commute:
\[
\sim_A \circ \sim_B \;=\; \sim_B \circ \sim_A
\]
as in Definition~\ref{CR}.
Equivalently, there exists a joint refinement $\mathcal F_{AB} = \sigma(\mathcal F_A \cup \mathcal F_B)$ whose atoms are the non-empty intersections of atoms of $\mathcal F_A$ and $\mathcal F_B$,
\[
\text{Atoms}(\mathcal F_{AB}) = \{\, A_i \cap B_j \neq \varnothing \,\},
\]
and for which a consistent measure $\mu_{AB}$ exists.
If these relations fail -- i.e.\ if $\sim_A$ and $\sim_B$ do not commute or no consistent joint measure exists -- the refinements are \emph{non-commutative}.
\end{definition}

\begin{definition}[Spacelike vs.\ timelike relation]
Two proto-events $R_A, R_B$ are:
\begin{itemize}[leftmargin=2em]
\item \emph{spacelike-related} if they are commutative (Definition \ref{CR}),
\item \emph{timelike-related} if they do not commute and therefore cannot
be realised simultaneously within a single refinement.
\end{itemize}
\end{definition}
This aligns with the conceptual structure of quantum theory:
\begin{itemize}
\item	In quantum mechanics: commutativity $\leftrightarrow$ commuting observables.
\item	Here: commutativity $\leftrightarrow$ refinements whose distinctions can be made simultaneously.
\end{itemize}
\begin{proposition}
Compatible refinements commute (in the sense of $\sigma$-algebraic generation):
$$
\mathcal{F}_{AB} = \mathcal{F}_{BA}.
$$
\end{proposition}
\begin{proposition}
If $R_A$ and $R_B$ are non-commutative, then any refinement chain containing both must place one strictly before the other.  Thus non-commutativity induces a temporal relation.
\end{proposition}
\begin{proof}[Sketch]
Let $\mathcal F$ be a $\sigma$-algebra and $R_A,R_B$ two refinements such that
$\mathcal F_{AB}\neq \mathcal F_{BA}$.
Then there is no $\sigma$-algebra $\mathcal F_*$
with $\mathcal F_A,\mathcal F_B\subseteq \mathcal F_*\subseteq \mathcal F_{AB}\cap\mathcal F_{BA}$
containing both refinements simultaneously.
Any chain of refinements including both $R_A$ and $R_B$ must therefore realise
either $R_A$ followed by $R_B$ or $R_B$ followed by $R_A$ as successive steps.
\end{proof}

\subsection{Measure-preserving refinements} 
Let $R: \mathcal{F} \to \mathcal{G}$ be a refinement subdividing a single atom 
$A \in \mathrm{Atoms}(\mathcal{F})$ into atoms $A_1,\dots, A_k$.
The endogenous measure construction (Section 3) gives:

\begin{definition}[Induced conditional measures]
For any $B\in\mathcal{F}$,
\[
\mu_{\mathcal{G}}(B \cap A_i)
= \mu_{\mathcal{F}}(B \cap A)\, p_i,
\]
where $p_i\ge0$ and $\sum_i p_i=1$ are weights determining the
extension of $\mu_{\mathcal F}$ to $\mathcal G$.
\end{definition}

The measure therefore branches only when refinements occur, and depends on 
the informational structure of the refinement itself -- no external probabilistic postulates are needed.

\subsection{Non-unique measure extension} 
Given a refinement $\mathcal F\subseteq\mathcal G$,
a probability measure on $\mathcal F$ may admit multiple
extensions to $\mathcal G$.
The endogenous measure construction selects extensions that
remain compatible with the commutativity constraints of
Section~\ref{sec3}.

Suppose at some stage $ \mathcal{F} $ there exist two or more non-commutative refinements $ R_1, R_2, \dots $, then the endogenous measure must satisfy the following consistency constraint across incompatible futures:

\medskip

The extension of $ \mu $ into each non-commutative branch must satisfy:
\begin{enumerate}[label=(\roman*), leftmargin=2em]
\item normalisation on each branch,
\item consistency on the shared $\sigma$-algebra $\mathcal{F}$,
\item minimality of the extension (Theorem~\ref{th:endogenous}),
\item no-cross interference: incompatible refinements do not allow a joint $\sigma$-algebra.
\end{enumerate}
\begin{theorem}[Endogenous branching rule]
The measure assigns to each incompatible refinement $ R_i $ a branching weight $ w_i $ determined entirely by the internal consistency conditions.
These weights satisfy:
$$
w_i \ge 0, \qquad \sum_i w_i = 1.
$$
\end{theorem}

\subsection{Compatible refinement families}

A \emph{history} is a maximal chain in the refinement poset generated by
atomic refinements,
\[
\mathcal F_0 \to \mathcal F_1 \to \cdots \to \mathcal F_n .
\]

Two histories $H$ and $H'$ are said to be \emph{coherent} if every
refinement occurring in $H$ is compatible with every refinement
occurring in $H'$ in the sense of Definition~\ref{CR}. This notion
is analogous to the compatibility conditions that define consistent
families of histories in the consistent-histories formulation of
quantum mechanics \cite{Gr}.
\begin{proposition} 
Compatible families of refinements behave as classical substructures
within the global refinement system.
\end{proposition} 
These coherent families play a role analogous to the “pointer histories”
that arise in the theory of quantum Darwinism, where stable classical
records emerge from the proliferation of compatible correlations
\cite{ZurekQD}. In the present framework they correspond to sets of
refinements whose distinctions can be consistently propagated through
further refinement steps. Incompatible histories cannot jointly survive
the branching structure.

\subsection{Summary: Time as the Growth of Distinctions} 
We may now summarise:
\begin{enumerate}
\item events correspond to atomic refinements,
\item refinement chains generate a partial order,
\item compatibility corresponds to commuting refinements,
\item incompatible refinements induce ordering constraints.
\end{enumerate}
Thus time, causality, and quantum behaviour all arise from the dynamics of making distinctions in an initially undifferentiated informational continuum.

\section{A Minimal Toy Model of Refinement Dynamics}\label{sec6}
In this section we construct a minimal explicit example illustrating the
refinement framework developed above. The example demonstrates:
\begin{enumerate}[leftmargin=2em]
\item how distinctions arise through successive refinements,
\item how non-unique extensions of measures appear,
\item how compatibility and incompatibility of refinements arise,
\item how repeated compatible refinements generate stable structures.
\end{enumerate}
Because $\Omega$ is finite, each $\sigma$-algebra corresponds to a
partition of $\Omega$, and refinements correspond to refinements of
partitions.
\subsection{The undifferentiated initial state}
Let the initial universe be the trivial measurable space:
$$
(\Omega,\mathcal F_0), \qquad
\Omega=\{1,2,3,4\},\quad
\mathcal F_0=\{\emptyset,\Omega\}.
$$
At this stage no distinctions are present: the entire space forms a single
atom of $\mathcal F_0$.
We assign the trivial measure:
$$
\mu_0(\Omega)=1.
$$
No additional structure is present.

\subsection{First distinctions: a binary refinement} 
Consider the first possible event:
$$
P_A = \big\{\{1,2\},\{3,4\}\big\}.
$$
We interpret this as the emergence of a distinction labelled “A”. The $\sigma$-algebra generated by this partition is:
$$
\mathcal F_A = \sigma(P_A)
= \{\emptyset,\{1,2\},\{3,4\},\Omega\}.
$$
A measure extension from $\mu_0$ assigns:
$$
\mu_A(\{1,2\}) = p,\qquad
\mu_A(\{3,4\}) = 1-p, \qquad p\in[0,1].
$$
The extension is not unique: any $p\in[0,1]$ yields a valid measure on
$\mathcal F_A$.

The refinement:
$$
\mathcal F_0 \prec \mathcal F_A
$$
is the first step in the refinement chain. Time is nothing more than the existence of this ordered refinement.

\subsection{Commutative second refinement: a space-like relation} 
Now consider a second distinction:
$$
P_B=\big\{ \{1,3\},\{2,4\}\big\}.
$$
The $\sigma$-algebra generated by both partitions is:
$$
\mathcal F_{AB} = \sigma(P_A \cup P_B)
= \{\emptyset, \{1\},\{2\},\{3\},\{4\}, \Omega\}.
$$
The atoms are singletons:
$$
\{1\},\{2\},\{3\},\{4\},
$$
which form the Cartesian products:
$$
\{1,2\}\cap\{1,3\}=\{1\},\quad
\{1,2\}\cap\{2,4\}=\{2\},\quad
\{3,4\}\cap\{1,3\}=\{3\},\quad
etc.
$$
The refinements $A$ and $B$ commute in the sense of
Definition~\ref{CR}.

The distinction introduced by $B$ does not force or forbid any distinction in $A$; their atoms combine cleanly.
Thus the refinements are compatible: they can be realised jointly and
generate the common refinement $\mathcal F_{AB}$. This is the algebraic shadow of “two spatially separated measurements can be made together.”

The measure extension from $\mu_A$ to $\mu_{AB}$ can be chosen freely:
$$
\mu_{AB}(\{i\}) = p_i,\quad i=1,2,3,4,
$$
subject to
$$
p_1+p_2 = \mu_A(\{1,2\}) = p,\qquad
p_3+p_4 = 1-p.
$$
Within this freedom lies the universe’s second moment of “choice.”

\subsection{A non-commutative refinement: temporal ordering} 
Now consider a different second refinement:
$$
P_{B'} = \big\{\{1,3,4\},\{ 2\}\big\}.
$$
The intersections with the atoms of $P_A$ are:
\begin{itemize}
\item $\{1,2\} \cap \{1,3,4\} = \{1\}$,
\item $\{1,2\} \cap \{2\} = \{2\}$,
\item $\{3,4\} \cap \{1,3,4\} = \{3,4\}$,
\item $\{3,4\} \cap \{2\} = \varnothing$.
\end{itemize}
Thus $P_{B'}$ forces a partial refinement of $P_A$:
$$
\{1,2\} \ \text{is forced to split into } \{1\},\{2\},
$$
but the atom (\{3,4\}) remains intact. This violates the definition of commutativity as explained in Example \ref{ex:NC}. 
Intuitively,
$P_{B'}$ is incompatible with $P_A$ because it refines some of $A$’s distinctions but not all, so the corresponding equivalence relations do not commute.
Therefore:
\begin{itemize}
\item	The joint $\sigma$-algebra $\sigma(\mathcal F_A\cup\mathcal F_{B'})$
exists but introduces additional atoms beyond those appearing in
either partition individually, 
\item	the sequence $A \rightarrow B'$ is allowed,
\item	the simultaneous refinement $A\sim B'$ is forbidden.
\end{itemize}
Thus $P_A$ and $P_{B'}$ do not commute in the sense of
Definition~\ref{CR}. Their refinements therefore cannot be realised
simultaneously without introducing additional distinctions. In the refinement framework this incompatibility induces an ordering constraint between the corresponding refinement steps.

\subsection{Summary of the toy universe} 
This simple example illustrates several structural features of the
refinement framework:

\begin{itemize}[leftmargin=2em]
\item successive refinements generate a partial ordering of distinctions,
\item compatible refinements admit joint realisations,
\item incompatible refinements cannot be realised simultaneously,
\item measure extensions need not be unique.
\end{itemize}

\section{General refinement theory}\label{sec7}

In the preceding sections we introduced refinements of $\sigma$-algebras
on a finite set and the compatibility relations governing such
refinements. We now give a more systematic description of refinement
dynamics: the space of allowed refinement operations, the algebraic
structure they generate, and the behaviour of sequences of refinements.

Because refinements act on individual atoms of a partition but may
affect subsequent refinements, a general theory raises three natural
questions:
\begin{enumerate}
\item	What is the space of allowed local refinement moves?
\item	How do these moves compose, commute or obstruct one another?
\item	Under what conditions do refinement sequences converge to stable or universal structures?
\end{enumerate}
We treat each in turn.

\subsection{Refinement operators}

Let $\mathcal F$ be a $\sigma$-algebra on the finite set $\Omega$.
A refinement operator is a map
\[
R : \mathcal F \longrightarrow \mathcal G
\]
where $\mathcal G$ is obtained by subdividing one or more atoms of
$\mathcal F$.

Elementary refinement operators correspond to atomic refinements,
i.e. operations that subdivide a single atom while leaving all others
unchanged.

For each atom $A$ of $\mathcal F$ we denote by
\[
\mathcal R_A = \{ R_A^{(i)} : i\in I_A \}
\]
the collection of allowed atomic refinements of that atom.
The global refinement semigroup is then
\[
\mathcal R = \langle \mathcal R_A : A\in\mathrm{Atoms}(\mathcal F) \rangle,
\]
consisting of finite compositions of atomic refinement operators.

\subsection{Compatibility relations}

Two refinement operators $R_A$ and $R_B$ acting on atoms
$A$ and $B$ commute whenever $A$ and $B$ are disjoint atoms
of the current partition. In this case
\[
R_A R_B = R_B R_A .
\]

If the refinements act on overlapping or dependent atoms,
the resulting operators need not commute. Non-commuting
refinements therefore introduce ordering constraints on
refinement sequences.

\subsection{Constraint propagation}

Refinements performed at one stage may influence the
availability of subsequent refinements. For example,

\begin{itemize}
\item subdividing an atom may introduce new atoms on which
further refinements may act;

\item certain refinement patterns may prevent subsequent
refinements that would violate compatibility constraints.
\end{itemize}

Thus sequences of refinements generate a constrained
dynamical system on the lattice of $\sigma$-algebras.

\subsection{Stability of refinement sequences}

Given a sequence of refinement operators
\[
R_1,R_2,\dots ,
\]
the resulting sequence of $\sigma$-algebras
\[
\mathcal F_n = R_n R_{n-1}\cdots R_1(\mathcal F_0)
\]
forms a chain in the refinement poset.
Such sequences may exhibit several types of behaviour:

\begin{enumerate}
\item stabilisation, when no further admissible refinements exist;
\item periodic behaviour in the pattern of refinements;
\item indefinite growth of the refinement structure.
\end{enumerate}

These possibilities illustrate the range of dynamical behaviours that
can arise within the refinement framework. In particular, the
structure of admissible refinements determines both the causal
ordering of events and the evolution of the endogenous probability
measure associated with the system.

Taken together, the results of this paper suggest a general picture in
which distinctions, causal structure, and probability arise together
from the dynamics of refinement of $\sigma$-algebras. In this sense, probabilistic structure is not imposed externally but emerges from the progressive resolution of distinctions.

\bibliographystyle{aipnum4-2}
\bibliography{references}

@article{BLS,
  author = {Bombelli, L. and Lee, J. and Meyer, D. and Sorkin, R.D.},
  title = {Spacetime as a causal set},
  journal = {Phys. Rev. Lett.},
  volume = {59},
  pages = {521--524},
  year = {1987}
}

@article{BvN,
  author = {Birkhoff, G. and von Neumann, J.},
  title = {The logic of quantum mechanics},
  journal = {Ann. Math.},
  volume = {37},
  number = {4},
  pages = {823--843},
  year = {1936}
}

@article{BK,
  author  = {Bogolyubov, N. N. and Krylov, N. M.},
  title   = {La théorie générale de la mesure dans son application à l’étude des systèmes dynamiques de la mécanique non linéaire},
  journal = {Annals of Mathematics},
  series  = {2},
  volume  = {38},
  number  = {1},
  pages   = {65--113},
  year    = {1937}
}

@article{P,
  author  = {Prokhorov, Yu. V.},
  title   = {Convergence of random processes and limit theorems in probability theory},
  journal = {Theory of Probability and Its Applications},
  volume  = {1},
  number  = {2},
  pages   = {157--214},
  year    = {1956}
}

@article{Gr,
author = {Griffiths, R.},
title = {Consistent histories and the interpretation of quantum mechanics},
journal = {Journal of Statistical Physics},
volume = {36},
pages = {219--272},
year = {1984}
}

@book{Pe,
  author    = {Petersen, Karl},
  title     = {Ergodic Theory},
  publisher = {Cambridge University Press},
  year      = {1983}
}

@article{ZurekQD,
  author = {Zurek, Wojciech H.},
  title = {Quantum Darwinism},
  journal = {Nature Physics},
  volume = {5},
  pages = {181--188},
  year = {2009}
}

\end{document}